\documentclass[12pt]{amsart}
\usepackage[alphabetic]{amsrefs}
\usepackage{mathdots, blkarray}
\usepackage{multicol}
\usepackage{graphicx}
\usepackage{amscd}
\usepackage{upgreek}
\usepackage{stmaryrd}
\usepackage{longtable}
\usepackage[T1]{fontenc}
\usepackage{latexsym, amsmath, amssymb, amsthm}
\usepackage[svgnames]{xcolor}
\usepackage{rsfso}
\usepackage[mathscr]{eucal}
\usepackage{mathtools}
\usepackage{mathptmx}
\usepackage{titletoc}
\usepackage{wrapfig}
\usepackage{float}
\usepackage{xypic}
\usepackage{microtype}
\usepackage{dsfont}
\usepackage{xcolor}
\usepackage{color}
\usepackage[colorlinks = true,
            linkcolor  = DarkBlue,
            urlcolor   = DarkRed,
            citecolor  = DarkGreen]{hyperref}
\usepackage{hyperref}
\allowdisplaybreaks	
\usepackage[centering, includeheadfoot, hmargin=1.0in, tmargin=0.5in, 
  bmargin=0.6in, headheight=6pt]{geometry}
\newtheorem{theorem}{Theorem}[section]
\newtheorem{prop}[theorem]{Proposition}
\newtheorem{defn}[theorem]{\rm\textsc{Definition}}
\newtheorem{lem}[theorem]{Lemma}
\newtheorem{coro}[theorem]{Corollary}

\newtheorem{thm}[theorem]{Theorem}

\newtheorem{rem}[theorem]{\rm\textsc{Remark}}
\newtheorem{exam}[theorem]{\rm\textsc{Example}}
\newcommand{\ideal}[1]{\ensuremath{\left\langle #1 \right\rangle}}

\DeclareMathOperator{\GL}{GL}
\DeclareMathOperator{\dia}{diag}
	
\DeclareMathOperator{\SL}{SL}

\DeclareMathOperator{\tr}{tr}
\DeclareMathOperator{\Tr}{Tr}
\newcommand{\B}{\mathcal{B}} 
\newcommand{\C}{\mathbb{C}} 
 
\newcommand{\Z}{\mathbb{Z}} 
\newcommand{\N}{\mathbb{N}} 
 
\newcommand{\g}{\mathfrak{g}} 
\newcommand{\h}{\mathfrak{h}} 
 
\newcommand{\sll}{\germ{sl}} 
\newcommand{\gl}{\germ{gl}} 
\newcommand{\ut}{\germ{u}} 
\newcommand{\ra}{\longrightarrow}

\newcommand{\hbo}{$\hfill\Diamond$}

\begin{document}
\title{Invariant theory and coefficient algebras\\ of Lie algebras} 
\def\shorttitle{Invariant theory and coefficient algebras of Lie algebras}

\author{Yin Chen}
\address{School of Mathematics and Physics, Key Laboratory of ECOFM of 
Jiangxi Education Institute, Jinggangshan University,
Ji'an 343009, Jiangxi, China \& Department of Finance and Management Science, University of Saskatchewan, Saskatoon, SK, Canada, S7N 5A7}
\email{yin.chen@usask.ca}

\author{Runxuan Zhang}
\address{Department of Mathematics and Information Technology, Concordia University of Edmonton, Edmonton, AB, Canada, T5B 4E4}
\email{runxuan.zhang@concordia.ab.ca}

\begin{abstract}
The coefficient algebra of a finite-dimensional Lie algebra on a finite-dimensional representation is defined as the subalgebra generated by all coefficients of the corresponding characteristic polynomial. We explore connections between classical invariant theory and the coefficient algebras of finite-dimensional complex Lie algebras on some representations. Specifically, we prove that with respect to any symmetric power of the standard representation: (a) the coefficient algebra of the upper triangular solvable complex Lie algebra is isomorphic to the ring of symmetric polynomials; (b) the coefficient algebra of the general linear complex Lie algebra is isomorphic to the invariant ring of the general linear group with the conjugacy action on the full space of matrices; and (c) the coefficient algebra of the special linear complex Lie algebra can be generated by classical trace functions. As an application, we determine the characteristic polynomial of the special linear complex Lie algebra on its standard representation. 
\end{abstract}

\date{\today}
%\subjclass[2010]{13A50.}
\maketitle \baselineskip=16pt

%%%%%%%%%%%%%%%%%%%%%%%%%%%Contents%%%%%%%%%%%%%%%%%%%%%%%%
%\textcolor{blue}{\tableofcontents{}}
\dottedcontents{section}[1.16cm]{}{1.8em}{5pt}
\dottedcontents{subsection}[2.00cm]{}{2.7em}{5pt}
%\dottedcontents{subsubsection}[2.86cm]{}{3.4em}{5pt}

%%%%%%%%%%%%%%%%%%%%%%%%%%%Sections%%%%%%%%%%%%%%%%%%%%%%%%
\section{Introduction}
\setcounter{equation}{0}
\renewcommand{\theequation}
{1.\arabic{equation}}
\setcounter{theorem}{0}
\renewcommand{\thetheorem}
{1.\arabic{theorem}}

\subsection{Background} The characteristic polynomial of a square matrix undoubtedly occupies a central position among the various branches of modern algebra. The theory of characteristic polynomial of multiple square matrices dates back to the works of Dedekind and Frobenius on group determinants of finite nonabelian groups in the 19th century. More precisely, the group determinant of a finite group $G=\{1,g_1,\dots,g_n\}$ is a homogeneous polynomial in $x_0,x_1,\dots,x_n$ defined as
\begin{equation}
\label{eq1.1}
\upvarphi_G:=\det(x_0\cdot I+x_1\cdot \upvarphi(g_1)+\cdots+x_n\cdot \upvarphi(g_n))
\end{equation}
where $I$ denotes the identity matrix and $\upvarphi$ denotes the left regular representation of $G$. Factoring the group determinant into a product of irreducible polynomials and exploring connections between group determinants and the group structure theory are two fundamental tasks in this area, which paved the way for the representation theory of finite groups. 

The group determinant  of a finite group $G$ can be generalized to the notion of characteristic polynomial of $G$ on any finite-dimensional  representation $V$ via replacing the left regular representation $\upvarphi$ in (\ref{eq1.1}) by the  
group homomorphism from $G$ to $\GL(V)$. Recently, an analogue of  characteristic polynomials in Lie theory has been introduced by \cite{HZ19} and has captured the attention of many researchers; see for example, \cite{CCD19, JL22}, and \cite{GLW24}.
One of our motivational examples was \cite[Theorem 2]{HZ19} which provides a sufficient and necessary condition for when Lie algebras are solvable via properties of their characteristic polynomials. 

To better understand the characteristic polynomial of a Lie algebra $\g$, we shall introduce a concept of the coefficient algebra of $\g$ on a finite-dimensional representation $V$. The primary objective in this article is to determine the structures of the coefficient algebras for several important Lie algebras on their representations. Our approach exposes some unexpected connections  between classical invariant theory, characteristic polynomials, and coefficient algebras of finite-dimensional complex Lie algebras. 

\subsection{Characteristic polynomials and coefficient algebras}  

Let $k$ be a field of any characteristic and $\g$ be a finite-dimensional Lie algebra over $k$. Suppose $V$ denotes an  $m$-dimensional representation of $\g$ and let us fix a basis $\{g_1,g_2,\dots,g_n\}$ for $\g$. We denote by $[g_i]$ the $m\times m$ resulting matrix of $g_i$ on $V$ for $i=1,2,\dots,n$, and write $I$ for the identity matrix of degree $m$.  The \textit{characteristic polynomial} of $\g$ on $V$ with respect to the basis $\{g_1,g_2,\dots,g_n\}$ is defined as the following homogeneous polynomial of degree $m$ in $n+1$ variables $x_0,x_1,\dots,x_n$:
\begin{equation}
\label{ }
\upvarphi_{\g}(V):=\det\left(x_0\cdot I+x_1\cdot [g_1]+\cdots+x_n\cdot [g_n]\right).
\end{equation}
Note that this definition is well-defined because it is independent of the choice of basis in $V$; see for example, \cite[Section 2]{HZ19}.  Understanding the relationship between the characteristic polynomial of a Lie algebra and its underlying Lie algebraic structure is a fundamental and significant problem. Numerous studies have examined this topic through approaches such as character theory in Lie algebra representations (\cite{KKSW24}), the study of eigen-varieties, and Poincar\'e polynomials of Lie algebras (\cite{AKY21}). Some interesting results have been derived. For example, 
a Lie algebra $\g$ is solvable if and only if $\upvarphi_{\g}(V)$ is completely reducible for all $V$ (see \cite[Theorem 5.1]{HZ19}); and a solvable Lie algebra $\g$ is nilpotent if and only if the corresponding Poincar\'e polynomial is the constant 1 (see \cite[Corollary 6.1]{AKY21}). For further applications of characteristic polynomials, see \cite{HY24}.

In this article, we take a new perspective to understand the characteristic polynomial $\upvarphi_{\g}(V)$. In fact, since
$\upvarphi_{\g}(V)$ also can be reviewed as a monic polynomial of degree $m$ in the single variable $x_0$ over the polynomial ring $k[x_1,\dots,x_n]$, we can write
\begin{equation}
\label{ }
\upvarphi_{\g}(V)=x_0^m+c_1\cdot x_0^{m-1}+\cdots+c_{m-1}\cdot x_0+c_m
\end{equation}
for some homogeneous polynomials $c_1,c_2,\dots,c_m\in k[x_1,\dots,x_n]$. Essentially, understanding the polynomial $\upvarphi_{\g}(V)$ is, in a sense, equivalent to determining these coefficients $c_1,c_2,\dots,c_m$ and their relationships. This philosophy leads us to introduce the following concept. 

\begin{defn}{\rm
The $k$-subalgebra generated by $\{c_1,c_2,\dots,c_m\}$ in $k[x_1,\dots,x_n]$ is called the \textit{coefficient algebra} of $\g$ on $V$ with respect to the basis $\{g_1,g_2,\dots,g_n\}$, and is denoted by $B_{\g}(V):=k[c_1,c_2,\dots,c_m]$.
}\end{defn}

The main purpose of this article is to determine the algebraic structure of the coefficient algebra $B_{\g}(V)$ for several well-known finite-dimensional linear Lie algebras over the complex field $\C$ on all symmetric powers of these standard representations. The structure of the coefficient algebra $B_{\g}(V)$ will further enhance our understanding of the characteristic polynomial $\upvarphi_{\g}(V)$ and the Lie representation $(\g,V)$.

\subsection{Main results} 

To articulate our main results, we assume that $2\leqslant n\in\N^+$ and we shall explicitly characterize the coefficient algebras of the upper triangular matrix Lie algebra $\ut_n(\C)$, the general linear Lie algebra  $\gl_n(\C)$, and the special linear Lie algebra $\sll_n(\C)$. These three Lie algebras play fundamental roles in the representation theory of finite-dimensional complex  Lie algebras. The reason for this might stem from the following important facts: (1) (Lie's Theorem) any solvable Lie algebra $\h$ could be viewed as a subalgebra of $\ut_n(\C)$ for some $n$ (see \cite[Corollary 2.3]{Car05});  (2) (Ado's Theorem) any Lie algebra $\g$ also could be viewed as a subalgebra of $\gl_n(\C)$ under some representations (see \cite[Section 6.2]{Jac79}); and (3) $\sll_n(\C)$ provides a significant family of finite-dimensional complex simple Lie algebras (\cite[Theorem 4.25]{Car05}). 

Let $d\in\N^+$, $V=\C^n$ be the standard representation of $\gl_n(\C)$,  and let $S^d(\C^n)$ denote the $d$th symmetric power of $\C^n$. Thus for any Lie subalgebra $\g\subseteq \gl_n(\C)$, $S^d(\C^n)$ is also a finite-dimensional $\g$-representation. In particular, we identify $S^1(\C^n)$ with $\C^n$ itself. Define 
$$\{E_{ij}\mid 1\leqslant i,j\leqslant n\}$$
to be the standard basis for $\gl_n(\C)$, where $E_{ij}$ denotes the $n\times n$ matrix whose $(i,j)$-entry is 1 and others 0. We may choose some linear combinations of $E_{ij}$ to form a basis of $\g$.

The first main result is the following. 

\begin{thm}\label{thm1}
The coefficient algebra $B_{\ut_n(\C)}(S^d(\C^n))$ is isomorphic to the ring of symmetric polynomials $\C[x_1,x_2,\dots,x_n]^{S_n}$, where $S_n$ denotes the symmetric group of degree $n$ with the standard permutation action on $\{1,2,\dots,n\}$.
\end{thm}

To state the second result, we consider the general linear group $\GL_n(\C)$ of degree $n$ over $\C$ and the conjugacy action of $\GL_n(\C)$ on the underlying space of $\gl_n(\C)$. The corresponding invariant ring
$\C[\gl_n(\C)]^{\GL_n(\C)}$ is well-understood in classical invariant theory, which is also a polynomial algebra over $\C$. 
Our second result reveals an unexpected connection between the coefficient algebra of $\gl_n(\C)$ on $S^d(\C^n)$ and
the invariant ring $\C[\gl_n(\C)]^{\GL_n(\C)}$.

\begin{thm}\label{thm2}
The coefficient algebra $B_{\gl_n(\C)}(S^d(\C^n))$ is isomorphic to  $$\C[\gl_n(\C)]^{\GL_n(\C)}=\C[x_{ij}\mid 1\leqslant i,j\leqslant n]^{\GL_n(\C)}=\C[s_1,s_2,\dots,s_n],$$ where $s_i$ denotes the classical conjugacy invariant of degree $i$ in $x_{ij}$.
\end{thm}

Suppose that $\SL_n(\C)$ denotes the special linear group of degree $n$ over $\C$ acting on the underlying space of $\sll_n(\C)$ in the conjugacy way. The third main result below computes the coefficient algebra of $\sll_n(\C)$ on all symmetric powers $S^d(\C^n)$.

\begin{thm}\label{thm3}
The coefficient algebra $B_{\sll_n(\C)}(S^d(\C^n))$ is isomorphic to 
$\C[\sll_n(\C)]^{\SL_n(\C)}$, and it is also a polynomial algebra generated by the trace functions $\tr_2,\tr_3,\dots,\tr_n$.
\end{thm}

This theorem leads to an immediate application (Corollary \ref{coro4.5}), which determines the characteristic polynomial $\upvarphi_{\sll_n(\C)}(\C^n)$ of $\sll_n(\C)$ on the standard representation $\C^n$.  

\subsection{Layout and conventions} 
In Section \ref{sec2},  we begin with computing the coefficient algebra of $\sll_2(\C)$ on any representation $V$ (Corollary \ref{coro2.3}) and provide a perspective from invariant theory to understand the characteristic polynomials of a Lie algebra on its self-dual representations. We also prove that a Lie algebra is nilpotent if and only if the coefficient algebra on the adjoint representation is $\C$; see Proposition \ref{prop2.7}. Section \ref{sec3} is devoted to the proof of Theorem \ref{thm1}. We provide a detailed explanation on actions of Lie algebras on symmetric powers and several examples to illustrate the crucial role of classical invariant theory in determining characteristic polynomials and coefficient algebras. In Section \ref{sec4}, we prove Theorems \ref{thm2} and \ref{thm3}, calculating the coefficient algebras of $\gl_n(\C)$ and $\sll_n(\C)$ on all symmetric powers $S^d(\C^n)$ of  the standard representation $\C^n$.

Throughout this paper, $k$ denotes a field of any characteristic (unless otherwise specified), $\C$ denotes the complex field, and we write $\N$ for the set of all nonnegative integers and $\N^+$ for the set of all positive integers. 
For an element $A\in\g$, we denote by $[A]$ the resulting matrix of $A$ on a $\g$-representation provided no confusion arises. All representations are assumed to be finite-dimensional.  The symbol $k[V]^G=k[x_1,x_2,\dots,x_n]^G$ is standard in the invariant theory and it denotes the invariant ring of a group $G$ with respect to a representation $V$, where $x_1,x_2,\dots,x_n$ form a basis of the dual space $V^*$; see for example, \cite{CW11} or \cite{DK15} for a general reference on algebraic invariant theory  and  see \cite{CDG20, CT19, CZ23a, CZ23b} and \cite{CZ24} for some recent applications of invariant theory in Lie algebras and other  topics.

\section{$\sll_2(\C)$ and Nilpotent Lie Algebras} \label{sec2}
\setcounter{equation}{0}
\renewcommand{\theequation}
{2.\arabic{equation}}
\setcounter{theorem}{0}
\renewcommand{\thetheorem}
{2.\arabic{theorem}}

\noindent A new theory (concept) in Lie algebras usually begins with the study on $\sll_2(\C)$ and its representations; see \cite{CCZ21, Che25}, and \cite[Section 5]{CZ23b}. The characteristic polynomials of $\sll_2(\C)$ on all finite-dimensional representations have been explicitly computed in \cite[Theorem 4]{CCD19} and \cite{JL22}. Based on their calculations, this section completely determines the coefficient algebras of $\sll_2(\C)$  on any representations and of nilpotent Lie algebras on their adjoint representations. Additionally, we provide a viewpoint from invariant theory to understand the characteristic polynomials of a Lie algebra on its self-dual representations. 

\subsection{The Lie algebra $\sll_2(\C)$} \label{sec2.1}
Let's begin with setting $\g=\sll_2(\C)$ and fixing a basis of $\sll_2(\C)$ throughout this subsection: 
\begin{equation}
\label{eq2.1}
e_1=\begin{pmatrix}
   0   & 1   \\
    0  & 0 
\end{pmatrix},e_2=\begin{pmatrix}
   -1   & 0   \\
    0  & 1
\end{pmatrix}, e_3=\begin{pmatrix}
   0   & 0   \\
    1  & 0 
\end{pmatrix}.
\end{equation}
Suppose $V=\C^2$ denotes the standard $2$-dimensional representation of $\sll_2(\C)$ over $\C$.  The characteristic polynomial of $\sll_2(\C)$ on $\C^2$ is
\begin{eqnarray*}
\upvarphi_{\sll_2(\C)}(\C^2)&=&\det\left(x_0\cdot \begin{pmatrix}
   1   & 0   \\
    0  & 1
\end{pmatrix}+x_1\cdot \begin{pmatrix}
   0   & 1   \\
    0  & 0 
\end{pmatrix}+x_2\cdot \begin{pmatrix}
   -1   & 0   \\
    0  & 1
\end{pmatrix}+x_3\cdot \begin{pmatrix}
   0   & 0   \\
    1  & 0 
\end{pmatrix}\right)\\
&=&\det\begin{pmatrix}
  x_0 -x_2   &x_1    \\
   x_3   & x_0+x_2 
\end{pmatrix}\\
&=&x_0^2-(x_2^2+x_1x_3). 
\end{eqnarray*}
Hence, the coefficient algebra $B_{\sll_2(\C)}(\C^2)=\C[x_2^2+x_1x_3]$ is a polynomial subalgebra of $\C[x_1,x_2,x_3]$ with Krull dimension 1.

Moreover, this statement still holds if $V$ is any faithful irreducible representation of $\sll_2(\C)$.

\begin{prop}\label{prop2.1}
If $V$ is a faithful irreducible representation of $\sll_2(\C)$, then
$B_{\sll_2(\C)}(V)$ is a polynomial subalgebra of $\C[x_1,x_2,x_3]$ generated by $x_2^2+x_1x_3$.
\end{prop}

\begin{proof}
By \cite[Theorem 4]{CCD19}, we see that the characteristic polynomial $\upvarphi_{\sll_2(\C)}(V)$ can be written as 
a product of either $x_0$ or $x_0^2-a\cdot (x_2^2+x_1x_3)$ for some $a\in\N$. Since $V$ is faithful, it follows that
$$\upvarphi_{\sll_2(\C)}(V)=x_0^{\ell_0}\cdot \Big(x_0^2-a_1\cdot (x_2^2+x_1x_3)\Big)\cdots \Big(x_0^2-a_s\cdot (x_2^2+x_1x_3)\Big)$$
for some $\ell_0\in\N$ and $s,a_1,a_2,\dots,a_s\in\N^+$.  For $1\leqslant i\leqslant s$, let us define 
$$y_i:= a_i\cdot (x_2^2+x_1x_3).$$
Recall that the coefficients of a polynomial equation can be expressed as
elementary symmetric polynomials in their roots. Thus 
$$\upvarphi_{\sll_2(\C)}(V)=x_0^{\ell_0}\cdot \Big(x_0^{2s}-\upepsilon_1(y_1,\cdots, y_s)\cdot x_0^{2s-2}+\cdots+(-1)^s \upepsilon_s(y_1,\cdots, y_s)\Big)$$
where $\upepsilon_i$ denotes the elementary symmetric polynomial of degree $i$ for all $i\in\{1,\dots, s\}$. Note that each $y_i\in \C[x_2^2+x_1x_3]$ as $a_i\neq 0$. Hence, $B_{\sll_2(\C)}(V)\subseteq\C[x_2^2+x_1x_3]$. 

Conversely, note that $\upepsilon_1(y_1,\cdots, y_s)=y_1+y_2+\cdots+y_s$, thus the second coefficient of $\upvarphi_{\sll_2(\C)}(V)$
is $(a_1+\cdots+a_s)\cdot (x_2^2+x_1x_3).$ Hence,
$$(a_1+\cdots+a_s)\cdot (x_2^2+x_1x_3)\in B_{\sll_2(\C)}(V).$$
Since $a_1,a_2,\dots,a_s$ are all positive, it follows that $a_1+\cdots+a_s\neq 0$ and so
$$x_2^2+x_1x_3\in B_{\sll_2(\C)}(V)$$ because the characteristic of $\C$ is zero. This proves that 
$B_{\sll_2(\C)}(V)=\C[x_2^2+x_1x_3]$. 
\end{proof}

\begin{rem}\label{rem2.2}{\rm
Recall that the trivial one-dimensional representation $V_1$ of $\sll_2(\C)$ is the unique irreducible non-faithful representation. In particular, $\upvarphi_{\sll_2(\C)}(V_1)=x_0$ and $B_{\sll_2(\C)}(V_1)=\C$.
\hbo}\end{rem}

\begin{coro}\label{coro2.3}
Let $V$ be a finite-dimensional representation of $\sll_2(\C)$. Then
$B_{\sll_2(\C)}(V)$ is either $\C$ or the polynomial subalgebra of $\C[x_1,x_2,x_3]$ generated by $x_2^2+x_1x_3$.
\end{coro}

\begin{proof}  
By the definition of characteristic polynomials, given two representations $V$ and $W$ of $\g$, we see that  
 $\upvarphi_{\g}(V\oplus W)=\upvarphi_{\g}(V)\cdot \upvarphi_{\g}(W)$. Proposition \ref{prop2.1} and Remark \ref{rem2.2}, together with the fact that $\sll_2(\C)$ is completely reducible (i.e., each representation can be written as the direct sum of some irreducible representations), imply that $B_{\sll_2(\C)}(V)=\C$ if $V\cong V_1\oplus\cdots\oplus V_1$; otherwise  $B_{\sll_2(\C)}(V)=\C[x_2^2+x_1x_3]$.
\end{proof}

\begin{rem}\label{rem2.4}{\rm
More generally, the base change in a Lie algebra $\g$ induces a degree-preserving algebraic action on the polynomial ring $k[x_0,x_1,\dots,x_n]$. Specifically, if $P$ denotes the transition matrix between two bases of $\g$, then
$$\widetilde{P}:=\begin{pmatrix}
   1   & 0   \\
     0 & P^\top
\end{pmatrix}$$
can linearly transform $\{x_0,x_1,\dots,x_n\}$ to a new generating set $\{x_0,x_1',\dots,x_n'\}$ of $k[x_0,x_1,\dots,x_n]$ and then can be extended to a $k$-algebraic action on $k[x_0,x_1,\dots,x_n]$. Under this action, the characteristic polynomial with respect to the second basis of $\g$ is exactly equal to the image of the characteristic polynomial with respect to the first basis. 
\hbo}\end{rem}

Together with Proposition \ref{prop2.1}, Corollary \ref{coro2.3}, and Remark \ref{rem2.4} obtains the following result. 

\begin{coro}
Let $\B$ be a basis of $\sll_2(\C)$ and $V$ be a finite-dimensional representation of $\sll_2(\C)$. Then
$B_{\sll_2(\C)}(V)$ (with respect to $\B$) is either $\C$ or isomorphic to the polynomial subalgebra of $\C[x_1,x_2,x_3]$ generated by an irreducible quadratic polynomial.
\end{coro}

\subsection{Characteristic polynomials on self-dual representations} 

We will see that a representation $V$ of a Lie algebra and the dual $V^*$ correspond to the same coefficient algebra, through their characteristic polynomials might be different.  

\begin{prop}
Let $\g$ be an $n$-dimensional Lie algebra with a fixed basis $\B$ and $V$ be a finite-dimensional representation of $\g$. If  $V^*$ is the dual representation of $V$, then $B_{\g}(V^*)=B_{\g}(V)$.
\end{prop}

\begin{proof} 
By \cite[Proposition 2.2]{KKSW24}, we see that the characteristic polynomial 
$$\upvarphi_{\g}(V^*)(x_0,x_1,\dots,x_n)=(-1)^{\dim(V)}\cdot \upvarphi_{\g}(V)(-x_0,x_1,\dots,x_n),$$ which means that up to sign, the two characteristic polynomials $\upvarphi_{\g}(V^*)$ and $\upvarphi_{\g}(V)$ have the same coefficients. Hence, $B_{\g}(V^*)=B_{\g}(V)$.
\end{proof}

We would like to provide a point of view from invariant theory to better understand the characteristic polynomials of a Lie algebra $\g$ on self-dual representations. 

Suppose the characteristic of the ground field $k$ is not 2 and that $V$ is a finite-dimensional self-dual representation of $\g$, i.e., $V^\ast \cong V$. Then
$$\upvarphi_{\g}(V)(x_0,x_1,\dots,x_n)=\upvarphi_{\g}(V^*)(x_0,x_1,\dots,x_n)=(-1)^{\dim(V)}\cdot \upvarphi_{\g}(V)(-x_0,x_1,\dots,x_n).$$

\textsc{Subcase 1.} If $\dim(V)$ is even, then 
\begin{equation}
\label{ }
\upvarphi_{\g}(V)(x_0,x_1,\dots,x_n)=\upvarphi_{\g}(V)(-x_0,x_1,\dots,x_n).
\end{equation}
Namely, the characteristic polynomial $\upvarphi_{\g}(V)$ can be understood as an invariant polynomial under the action of the group $\ideal{\upsigma}$ on $k[x_0,x_1,\dots,x_n]$, where 
$$\upsigma:=\dia\{-1,1,\dots,1\}.$$
Note that the order of $\upsigma$ is 2. Moreover, $\upsigma(x_0)=-x_0$ and $\upsigma(x_i)=x_i$ for all $1=1,2,\dots,n$. Thus
$x_0^2,x_1,\dots,x_n$ form a homogeneous system of parameters by \cite[Lemma 2.6.3]{CW11}. Hence, it follows from \cite[Corollary 3.1.6]{CW11} or \cite[Proposition 16]{Kem96} that the invariant ring $k[x_0,x_1,\dots,x_n]^{\ideal{\upsigma}}$ is generated by $\{x_0^2,x_1,\dots,x_n\}$. This means that $\upvarphi_{\g}(V)$ must be a polynomial in $x_0^2,x_1,\dots,x_n$.

\textsc{Subcase 2.} Suppose $\dim(V)$ is odd and recall that the degree of $\upvarphi_{\g}(V)$ equals $\dim(V)$, we have 
\begin{equation}
\label{ }
\upvarphi_{\g}(V)(x_0,x_1,\dots,x_n)=-\upvarphi_{\g}(V)(-x_0,x_1,\dots,x_n)=\upvarphi_{\g}(V)(x_0,-x_1,\dots,-x_n).
\end{equation}
In other words, the characteristic polynomial $\upvarphi_{\g}(V)$, in this case, can be viewed as an invariant polynomial under the action of the group $\ideal{-\upsigma}$ on $k[x_0,x_1,\dots,x_n]$. 
Note that $k[x_0,x_1,\dots,x_n]^{\ideal{-\upsigma}}$ is a nonmodular invariant ring, thus the Noether's bound theorem (see \cite[Theorem 3.5.1]{CW11}) implies that $k[x_0,x_1,\dots,x_n]^{\ideal{-\upsigma}}$ is generated by invariant polynomials of degree at most $|\ideal{-\upsigma}|=2$. Since $(-\upsigma)(x_0)=x_0$ and $(-\upsigma)(x_i)=-x_i$ for all $1=1,2,\dots,n$, we see that $k[x_0,x_1,\dots,x_n]^{\ideal{-\upsigma}}$ is generated by 
$$\{x_0,x_ix_j\mid 1\leqslant i,j\leqslant n\}.$$
Therefore, $\upvarphi_{\g}(V)$ is a polynomial in $x_0$ and some $x_ix_j$, where $i,j\in\{1,2,\dots,n\}$.

\subsection{Nilpotent Lie algebras} 

We may obtain a sufficient and necessary description for when a complex Lie algebra $\g$ is nilpotent via the coefficient algebra of $\g$ on the adjoint representation.

\begin{prop}\label{prop2.7}
Let $\g$ be an $n$-dimensional complex Lie algebra with any basis and $V$ be the adjoint representation of $\g$. Then
$\g$ is nilpotent if and only if $B_{\g}(V)=\C$.
\end{prop}

\begin{proof} 
This essentially follows from \cite[Theorem 4.8]{KKSW24}, which states that $\g$ is nilpotent if and only if the corresponding characteristic polynomial is $x_0^n$.
\end{proof}

\begin{rem}{\rm
Proposition \ref{prop2.7} also was proved in \cite[Theorem 4.8]{KKSW24} by using the Engel's theorem. In fact, by Engel's theorem, $\g$ is nilpotent if and only if $\g$ has a basis with respect to which each element of $\g$ on the adjoint representation $V$  is an upper triangular matrix along with diagonal entries 0; see for example, \cite[Corollary 2.11]{Car05}. By Remark \ref{rem2.4}, the latter statement is equivalent to saying that $B_{\g}(V)=\C$. \hbo
}\end{rem}

Moreover, the Lie's theorem (see \cite[Corollary 2.3]{Car05}) is also greatly helpful in computing the
characteristic polynomials and the coefficient algebras of some solvable Lie algebras, and we believe, especially for the cases of high-dimensional  Lie algebras and their representations.

Let us close this section with the following example that illustrates how we can use the Lie's theorem to calculate the
coefficient algebra of a solvable Lie algebra. 

\begin{exam}{\rm
Suppose that $\g$ denotes the 3-dimensional subalgebra of $\gl_3(\C)$ consisting of all matrices of the following form
$$\begin{pmatrix}
   0   & a&b   \\
   -a   &0&c\\
   0&0&0  
\end{pmatrix}$$
where $a,b,c\in\C$. Then $\g$ is a solvable but not nilpotent Lie algebra. Let $V=\C^3$ be the standard representation of $\g$. 
We define
\begin{equation}
\label{ }
X:=\begin{pmatrix}
  0 & x_1&x_2   \\
   -x_1   &0&x_3\\
   0&0&0
\end{pmatrix}
\end{equation}
A direct calculation shows  that the characteristic polynomial
$$\upvarphi_{\g}(\C^3)=\det (x_0\cdot I+X)=\det\begin{pmatrix}
   x_0   & x_1&x_2   \\
   -x_1   &x_0&x_3\\
   0&0&x_0 
\end{pmatrix}=x_0^3+x_1^2\cdot x_0.$$
Hence, the coefficient algebra $B_{\g}(\C^3)=\C[x_1^2]$. 

On the other hand, using Lie's theorem we may prove that there exists an upper triangular matrix in $\gl_3(\C[x_1,x_2,x_3])$ such that it is similar with $X$. By \cite[Corollary 16.4]{Bro93}, the matrix $x_0\cdot I+X$ must be equivalent to an upper triangular matrix over $\C[x_1,x_2,x_3][x_0]$. In other words, after some elementary row/column transformations over $\C[x_0,x_1,x_2,x_3]$, the matrix $x_0\cdot I+X$ could be upper triangularized.   
Suppose that $U\in \gl_3(\C[x_0,x_1,x_2,x_3])$ denotes an equivalent upper triangular matrix of $x_0\cdot I+X$. Then there exist two products $P,Q$ of elementary matrices such that $$U=P\cdot (x_0\cdot I+X)\cdot Q.$$ Hence,
$$\det(x_0\cdot I+X)=\frac{\det(U)}{\det(P)\cdot \det(Q)}.$$
In our case, we let $P_1$ be the transformation of multiplying the first row with $x_1$, 
$P_2$ the transformation of multiplying the second row with $x_0$, and $P_3$ be the transformation of adding the first row to the second row. Setting $P=P_3P_2P_1$ and $Q=I$, we obtain
$$U=\begin{pmatrix}
   x_0x_1   & x_1^2&x_1x_2   \\
   0   &x_0^2+x_1^2&x_0x_3+x_1x_2\\
   0&0&x_0 
\end{pmatrix}.$$
Therefore, $\det(x_0\cdot I+X)=\frac{\det(U)}{\det(P)\cdot \det(Q)}=x_0^3+x_1^2\cdot x_0$, and  $B_{\g}(\C^3)=\C[x_1^2]$.\hbo
}\end{exam}

\section{Coefficient Algebras of $\ut_n(\C)$} \label{sec3}
\setcounter{equation}{0}
\renewcommand{\theequation}
{3.\arabic{equation}}
\setcounter{theorem}{0}
\renewcommand{\thetheorem}
{3.\arabic{theorem}}

\noindent This section and next section will work over the complex field $\C$ and we assume that $2\leqslant n\in\N^+$. 
The goal of this section is to determine the coefficient algebra of the upper triangular matrix Lie algebra $\ut_n(\C)$ on all symmetric power $S^d(\C^n)$ induced by the standard representation $\C^n$. 

\subsection{Actions on  $S^d(V)$} 

Suppose that $\C^n$ denotes the $n$-dimensional standard representation of the general linear Lie algebra $\gl_n(\C)$ and $S^d(\C^n)$ denotes the $d$th symmetric power representation induced by $\C^n$ for $d\in\N^+$. We identify $S^1(\C^n)$ with $\C^n$ and fix the standard basis  $\{E_{ij}\mid 1\leqslant i,j\leqslant n\}$ for $\gl_n(\C)$. 

Let $\g$ be a subalgebra of $\gl_n(\C)$. Clearly, $\C^n$ is also the standard representation of $\g$.
Suppose that $\{e_1,e_2,\dots,e_n\}$ denotes the standard basis of $\C^n$. Then the following set
\begin{equation}
\label{eq3.1}
\B:=\left\{e_1^{d_1}e_2^{d_2}\cdots e_n^{d_n}\mid d_1+d_2+\cdots+d_n=d,\textrm{ all }d_i\in\N\right\}
\end{equation}
is a basis of $S^d(\C^n)$; see \cite[Exercise B.12, page 476]{FH91}.  This basis $\B$ also can be described as
\begin{equation}
\label{ }
\B=\left\{e_{i_1}e_{i_2}\cdots e_{i_d}\mid 1\leqslant i_1\leqslant i_2\leqslant \dots\leqslant i_d\leqslant n\right\}.
\end{equation}
Thus
\begin{equation}
\label{ }
\dim(S^d(\C^n))=\binom{n+d-1}{d}.
\end{equation}
We endow $S^d(\C^n)$ with the lexicographic monomial ordering with $e_1>e_2>\cdots >e_n$. With this ordering, the basis $\B$ becomes ordered and it can be written as follows:
\begin{equation}
\label{eq3.4}
\B=\left\{e_1^d,e_1^{d-1}e_2,\dots,e_1^{d-1}e_n,e_1^{d-2}e_2^2,\dots,e_n^d\right\}.
\end{equation}

Moreover, the action of $\g$ on $\C^n$ can be extended as a $\g$-representation on $S^d(\C^n)$ given by
\begin{equation}
\label{eq3.5}
g\cdot(v\cdot w):=g(v)\cdot w+v\cdot g(w)
\end{equation}
for $g\in\g, v$ and $w\in S^d(\C^n)$. In particular, for all $e_{i_1}e_{i_2}\cdots e_{i_d}\in \B$, we have
\begin{equation}
\label{eq3.6}
g\cdot (e_{i_1}e_{i_2}\cdots e_{i_d})=\sum_{s=1}^d e_{i_1}\cdots e_{i_s-1}\cdot g(e_{i_s})\cdot e_{i_s+1}\cdots e_{i_d}.
\end{equation}
Hence, together with the fact that $E_{ij}(e_s)=\updelta_{si}\cdot e_j$, we see that
\begin{equation}
\label{eq3.7}
E_{ij}\cdot (e_{i_1}e_{i_2}\cdots e_{i_d})=\sum_{s=1}^d \updelta_{i_s,i}\cdot  e_{i_1}\cdots e_{i_s-1}\cdot e_j\cdot e_{i_s+1}\cdots e_{i_d}
\end{equation}
where $\updelta_{ij}$ denotes the Kronecker delta.

We denote by $[E_{ij}]$ the resulting matrix of $E_{ij}$ on $S^d(\C^n)$ with respect to the ordered basis $\B$ in (\ref{eq3.4}) if no confusion arises.

\begin{exam}\label{exam3.1}
{\rm
Suppose that $n=d=2$. Then $\dim(S^2(\C^2))=3$ and $\{e_1^2,e_1e_2,e_2^2\}$ is the ordered basis of $S^2(\C^2)$. 
Consider the basis elements of $\gl_2(\C)$:
$$E_{11}=\begin{pmatrix}
   1   & 0   \\
   0   & 0 
\end{pmatrix}\textrm{ and  }E_{12}=\begin{pmatrix}
   0   & 1   \\
   0   & 0 
\end{pmatrix}$$
where $E_{11}(e_1)=e_1, E_{11}(e_2)=0$, and $E_{12}(e_1)=e_2, E_{12}(e_2)=0$. Thus
\begin{eqnarray*}
E_{11}(e_1^2) & = & E_{11}(e_1)\cdot e_1+e_1\cdot E_{11}(e_1)= 2e_1^2\\
E_{11}(e_1e_2) & = & E_{11}(e_1)\cdot e_2+e_1\cdot E_{11}(e_2)=e_1e_2\\
E_{11}(e_2^2) & = & E_{11}(e_2)\cdot e_2+e_2\cdot E_{11}(e_2)=0.
\end{eqnarray*}
This means that with respect to $\{e_1^2,e_1e_2,e_2^2\}$, $[E_{11}]$ is the diagonal matrix $\dia\{2,1,0\}$. Similarly, 
\begin{eqnarray*}
E_{12}(e_1^2) & = & E_{12}(e_1)\cdot e_1+e_1\cdot E_{12}(e_1)= 2e_1e_2\\
E_{12}(e_1e_2) & = & E_{12}(e_1)\cdot e_2+e_1\cdot E_{12}(e_2)=e_2^2\\
E_{12}(e_2^2) & = & E_{12}(e_2)\cdot e_2+e_2\cdot E_{12}(e_2)=0.
\end{eqnarray*}
Hence, $[E_{12}]=\begin{pmatrix}
    0  & 2&0   \\
    0  & 0&1\\
    0&0&0 
\end{pmatrix}.$
Similarly, we have
$[E_{21}]=\begin{pmatrix}
    0  & 0&0   \\
    1  & 0&0\\
    0&2&0 
\end{pmatrix}$
and $[E_{22}]=\dia\{0,1,2\}$.
\hbo}\end{exam}

\subsection{Proof of Theorem \ref{thm1}} 
This subsection is devoted to the proof of Theorem \ref{thm1}. 
Before giving the proof, we first establish several necessary lemmas.

\begin{lem}\label{lem3.2}
Let $x_0,x_1,\dots,x_n$ be $n+1$ variables. Then 
\begin{equation}
\label{ }
\det\left(x_0\cdot I+\sum_{i=1}^nx_i\cdot [E_{ii}]\right)=\prod^{d_i\in\N}_{d_1+\dots+d_n=d}\left(x_0+\sum_{i=1}^nd_i\cdot x_i\right).
\end{equation}
\end{lem}

\begin{proof}
Let us consider the ordered basis $\B$ described as (\ref{eq3.4}) and $i\in\{1,2,\dots,n\}$. Since $E_{ii}(e_s)=\updelta_{si}\cdot e_i$, it follows from (\ref{eq3.7}) that
$$E_{ii}\cdot (e_{i_1}e_{i_2}\cdots e_{i_d})=\sum_{s=1}^d \updelta_{i_s,i}\cdot  e_{i_1}\cdots e_{i_s-1}\cdot e_i\cdot e_{i_s+1}\cdots e_{i_d}.$$
This implies that $[E_{ii}]$ is a diagonal matrix of size $\binom{n+d-1}{d}$ over $\Z$ and the diagonal elements of 
$[E_{ii}]$ count occurrences of $e_i$ in the basis elements of $\B$. Thus, the matrix 
$$\sum_{i=1}^nx_i\cdot [E_{ii}]$$
is also diagonal of the same size but over the polynomial ring $\Z[x_1,\dots,x_n]$. In particular, 
each diagonal entry of $\sum_{i=1}^nx_i\cdot [E_{ii}]$ has the following form
$$d_1x_1+d_2x_2+\dots+d_nx_n$$
where all $d_i\in\N$ and $d_1+d_2+\dots+d_n=d$. Since the diagonal elements of $[E_{ii}]$ count occurrences of $e_i$, we see that any two entries in $\sum_{i=1}^nx_i\cdot [E_{ii}]$ are different. Hence, this fact, together with the fact that 
$$|\B|=\binom{n+d-1}{d}=\dim(S^d(\C^n))=\textrm{ the size of the matrix }\sum_{i=1}^nx_i\cdot [E_{ii}]$$
implies that there exists a one-to-one correspondence between the set $\B$ expressed in (\ref{eq3.1}) and the set of all diagonal entries of the matrix $\sum_{i=1}^nx_i\cdot [E_{ii}]$.
Therefore,
$$\det\left(x_0\cdot I+\sum_{i=1}^nx_i\cdot [E_{ii}]\right)=\prod^{d_i\in\N}_{d_1+\dots+d_n=d}\left(x_0+\sum_{i=1}^nd_i\cdot x_i\right).$$
This completes the proof.
\end{proof}

\begin{lem} \label{lem3.3}
For any $d\geqslant 2$ and $1\leqslant i<j\leqslant n$, each $[E_{ij}]$ is a strictly upper triangular matrix. 
\end{lem}

\begin{proof}
Note that $j>i$, thus it follows from (\ref{eq3.7}) that the action of $E_{ij}$ shifts up the indices of the basis monomials in the basis $\B$. This fact, together with the ordering of $\B$ described as in (\ref{eq3.4}), implies that the resulting matrix of
$E_{ij}$ is strictly upper triangular. 
\end{proof}

The following lemma is well-known in invariant theory of finite groups; see for example, \cite[Section 1.2, page 7]{CW11}.

\begin{lem} \label{lem3.4}
Let $k[W]^G$ be the invariant ring of a finite group $G$ with respect to a $G$-representation $W$.  For $f\in k[W]$, write $G(f)$ for the $G$-orbit $\{\upsigma(f)\mid \upsigma\in G\}$ of $f$ and define $m:=|G(f)|$. Then 
\begin{equation}
\label{eq3.9}
\prod_{h\in G(f)} (\uplambda-h)=\sum_{i=0}^m(-1)^i\cdot c_i\cdot \uplambda^{m-i}
\end{equation} 
where $\uplambda$ is an indeterminant and all coefficients $c_i\in k[W]^G$.
\end{lem}

Suppose $i\in\N^+$. Let us recall the elementary symmetric polynomials $\upepsilon_i$ and power sums $p_i$ in the variables $x_1,x_2,\dots,x_n$, which will play a key role in the proofs of the main theorems: 
\begin{eqnarray*}
\upepsilon_i(x_1,x_2,\dots,x_n) & := & \sum_{1\leqslant r_1<r_2<\cdots<r_j\leqslant n} x_{r_1}x_{r_2}\cdots x_{r_j} \\
p_i(x_1,x_2,\dots,x_n) &:=&x_1^i+x_2^i+\cdots+x_n^i.
\end{eqnarray*}
These polynomials $\upepsilon_i$ and $p_i$ (for all $i\in\N^+$) are symmetric polynomials; the structure theorem of the algebra of symmetric polynomials asserts that they can generate the symmetric invariant ring $\C[x_1,x_2,\dots,x_n]^{S_n}$; see for example, \cite[Section 1.1]{Stu08}. 

 \begin{lem}\label{lem3.5}
Let $k$ be a field of characteristic zero and $S_n$ be the symmetric group of degree $n$ acting the polynomial ring $k[x_1,x_2,\dots,x_n]$ by permuting the variables in the standard way. Then $k[x_1,x_2,\dots,x_n]^{S_n}$ is a polynomial algebra, generated by the elementary symmetric polynomials $\upepsilon_1,\upepsilon_2,\dots,\upepsilon_n$ or the power sums $p_1,p_2,\dots,p_n$, i.e.,
$$k[x_1,x_2,\dots,x_n]^{S_n}=k[\upepsilon_1,\upepsilon_2,\dots,\upepsilon_n]=k[p_1,p_2,\dots,p_n].$$
\end{lem}

\begin{rem}\label{rem3.6}{\rm
Note that when $n>1$, generally, we have $\upepsilon_i\neq p_i$ unless $i=1$. A method of expressing each
$p_i$ in terms of $\upepsilon_1,\upepsilon_2,\dots,\upepsilon_n$ (or vice versa) can be found in \cite[Section 1.2]{Mac95}.
\hbo}\end{rem}

We are ready to prove our first main result. 

\begin{proof}[Proof of Theorem \ref{thm1}]
We first consider the case of $d=1$, i.e., the case of standard representation $\C^n$. In this case, recall that the dimension of $\ut_n(\C)$ is $\frac{n(n+1)}{2}$ and  the standard basis of $\ut_n(\C)$ could be chosen as $\{E_{ij}\mid 1\leqslant i\leqslant j\leqslant n\}$. Thus  the characteristic polynomial is
$$\upvarphi_{\ut_n(\C)}(\C^n)=\det\left(x_0\cdot I+\sum_{i=1}^n x_{i}\cdot E_{ii}+\sum_{1\leqslant i<j\leqslant n} x_{ij}\cdot E_{ij}\right)$$
for some variables $x_0,x_i,$ and $x_{ij}$. Namely,
$$\upvarphi_{\ut_n(\C)}(\C^n)=\det\left(x_0\cdot I+\begin{pmatrix}
    x_{1}  &x_{12}&\cdots&x_{1n}    \\
    0  &x_{2}&\ddots&\vdots\\  
    \vdots&\ddots&\ddots&x_{n-1,n}\\
    0&\cdots&0&x_{n}
\end{pmatrix}\right)=\prod_{i=1}^n(x_0+x_{i}).$$
Therefore, the coefficient algebra $B_{\ut_n(\C)}(\C^n)$ is the polynomial algebra generated by
$n$ elementary polynomials in $x_{1},x_{2},\dots,x_{n}$, i.e., 
$$B_{\ut_n(\C)}(\C^n)=\C[\upepsilon_1,\upepsilon_2,\dots,\upepsilon_n]=\C[x_1,x_2,\dots,x_n]^{S_n}.$$

Now suppose that $d\geqslant 2$ and let us fix the ordered basis $\B$ in (\ref{eq3.4}) for 
the symmetric power $S^d(\C^n)$. Then the characteristic polynomial of 
$\ut_n(\C)$ on $S^d(\C^n)$ is
$$\upvarphi_{\ut_n(\C)}(S^d(\C^n))=\det\left(x_0\cdot I+\sum_{i=1}^n x_{i}\cdot [E_{ii}]+\sum_{1\leqslant i<j\leqslant n} x_{ij}\cdot [E_{ij}]\right).$$
It follows from Lemmas \ref{lem3.3} and \ref{lem3.2} that
\begin{equation}
\label{ }
\upvarphi_{\ut_n(\C)}(S^d(\C^n))=\det\left(x_0\cdot I+\sum_{i=1}^n x_{i}\cdot [E_{ii}]\right)=\prod^{d_i\in\N}_{d_1+\dots+d_n=d}\left(x_0+\sum_{i=1}^nd_i\cdot x_i\right).
\end{equation}

The standard permutation action of $S_n$ on $\{1,2,\dots,n\}$ induces an action of $S_n$ on the basis $\B$ in (\ref{eq3.1}) by permuting indices of the monomial basis elements. Suppose that 
\begin{equation}
\label{decomp}
\B=\B_1\cup\B_2\cup\cdots\cup\B_t
\end{equation}
denotes the disjoint union of all orbits. Note that each element in $\B$ corresponds to a linear form:
$$x_0+\sum_{i=1}^nd_i\cdot x_i.$$
For $j\in\{1,2,\dots,t\}$, we write $\upvarphi_j$ for the the product of all linear forms corresponding to the elements of $\B_j$. For example, we may assume that $\B_1=\{e_1^d,e_2^d,\dots,e_n^d\}$ and
$$\upvarphi_1=\prod_{i=1}^n(x_0+d\cdot x_i)=x_0^{n}+d\upepsilon_1\cdot x_0^{n-1}+\cdots+d^{n-1}\upepsilon_{n-1}\cdot x_0+d^n\upepsilon_n.$$
Moreover, by Lemma \ref{lem3.4}, we see that all coefficients obtained by expanding $\upvarphi_j$ along $x_0$
are symmetric polynomials in $x_1,x_2,\dots,x_n$. Since
$$\upvarphi_{\ut_n(\C)}(S^d(\C^n))=\prod_{j=1}^n\upvarphi_j,$$ 
it follows that the coefficient algebra $B_{\ut_n(\C)}(S^d(\C^n))$ (recall that it is generated by all coefficients of $\upvarphi_{\ut_n(\C)}(S^d(\C^n))$)  is contained in $\C[x_1,x_2,\dots,x_n]^{S_n}=\C[\upepsilon_1,\upepsilon_2,\dots,\upepsilon_n]$. 

To complete the proof, we need to prove that the inverse containment, i.e., 
$\C[\upepsilon_1,\upepsilon_2,\dots,\upepsilon_n]\subseteq B_{\ut_n(\C)}(S^d(\C^n))$. By Lemma \ref{lem3.5}, it suffices to show
$$\C[p_1,p_2,\dots,p_n]\subseteq B_{\ut_n(\C)}(S^d(\C^n)).$$
In fact, this statement can be proved by exhibiting the first $n+1$ coefficients $c_0,c_1,c_2,\dots,c_n$ of $\upvarphi_{\ut_n(\C)}(S^d(\C^n))$. First of all, it is clear that 
$c_0=1$ because $\upvarphi_{\ut_n(\C)}(S^d(\C^n))$ is a monic polynomial. Secondly, for $c_1$, there exists a positive number $a_1\in\N^+$ such that $c_1=a_1\upepsilon_1=a_1p_1$. Precisely, for each $j\in\{1,2,\dots,t\}$, since all entries of a partition of $d$ are nonnegative, the second coefficient of $\upvarphi_j$ is a positive scalar multiple of $\upepsilon_1$. Note that $c_1$ is the sum of the second coefficients of all $\upvarphi_j$, thus it is also a positive scalar multiple of $p_1$.
Moreover, since the characteristic of $\C$ is zero, it follows that
$$p_1\in B_{\ut_n(\C)}(S^d(\C^n)).$$
Let $m:=\binom{n+d-1}{d}$. Now we may write
\begin{equation}
\label{eq3.12}
\upvarphi_{\ut_n(\C)}(S^d(\C^n))=x_0^{m}+a_1p_1\cdot x_0^{m-1}+c_2\cdot x_0^{m-2}+c_3\cdot x_0^{m-3}+\cdots.
\end{equation}
For each $j\in\{2,\dots,n\}$, note that $c_j$ is a homogeneous symmetric polynomial of degree $j$, thus it can be expressed as a polynomial in $p_1,\dots,p_{j}$. Choose the graded lexicographic monomial ordering with $x_1>x_2>\cdots>x_n$ on $\C[x_1,\dots,x_n]$. Since $c_j$ is the $j$th elementary symmetric polynomial in all linear forms corresponding to all partitions of $d$, we see that the leading monomials of $c_j$ and $p_j$ both are $x_1^j$. 
Hence, there are a polynomial $Q$ in $j-1$ variables and a nonzero scalar $a_j$ such that
\begin{equation}
\label{ }
c_j=a_jp_j+Q(p_1,\dots,p_{j-1}),
\end{equation}
which, together with the fact that $p_1\in B_{\ut_n(\C)}(S^d(\C^n))$ and $a_2\neq 0$, implies that 
$$p_2\in B_{\ut_n(\C)}(S^d(\C^n)).$$
Proceeding in this way, we see that
$p_3,\dots,p_{n}\in B_{\ut_n(\C)}(S^d(\C^n))$ as well. Therefore, the coefficient algebra
$B_{\ut_n(\C)}(S^d(\C^n))=\C[x_1,x_2,\dots,x_n]^{S_n}$.
\end{proof}

\subsection{Examples}
In this subsection, we present some low-dimensional examples to illustrate how to explicitly calculate the coefficient algebra $B_{\ut_n(\C)}(S^d(\C^n))$.

\begin{exam}[$n=2$]{\rm
(1) Let $d=1$.  The characteristic polynomial $\upvarphi_{\ut_2(\C)}(\C^2)=(x_0+x_{1})(x_0+x_{2})$ and so
the coefficient algebra of $\ut_2(\C)$ on $\C^2$ is $\C[x_{1}+x_{2},x_{1}x_{2}]=\C[x_1,x_2]^{S_2}$.

(2) Consider $d=2$. By Example \ref{exam3.1}, we see that the characteristic polynomial is
\begin{eqnarray*}
\upvarphi_{\ut_2(\C)}(S^2(\C^2))&=&\det\left(x_0\cdot I+x_1\cdot [E_{11}]+x_2\cdot [E_{22}]+x_{12}\cdot [E_{12}]\right)\\
&=&\det\begin{pmatrix}
   x_0+2x_1   & 2x_{12}& 0  \\
    0  &  x_0+x_1+x_2&x_{12}\\
     0 &0&x_0+2x_2
\end{pmatrix}\\
&=&(x_0+2x_1)(x_0+2x_2)(x_0+x_1+x_2).
\end{eqnarray*}
For this case, there are two orbits in the orbit decomposition (\ref{decomp}). 
Note that the orbit of $2x_1$ under the permutation action of $S_2$ is $\{2x_1,2x_2\}$, thus
$$\upvarphi_1=(x_0+2x_1)(x_0+2x_2)=x_0^2+2\upepsilon_1 \cdot x_0+4\upepsilon_2.$$
Similarly, as $x_1+x_2=\upepsilon_1$ is invariant of $S_2$, its orbit only consists of itself and thus
$$\upvarphi_2=x_0+\upepsilon_1.$$
The characteristic polynomial is
\begin{eqnarray*}
\upvarphi_{\ut_2(\C)}(S^2(\C^2))&=&\upvarphi_1\cdot \upvarphi_2\\
&=&(x_0^2+2\upepsilon_1 \cdot x_0+4\upepsilon_2)(x_0+\upepsilon_1)\\
&=&x_0^3+2\upepsilon_1 \cdot x_0^2+4\upepsilon_2\cdot x_0+\upepsilon_1\cdot x_0^2+2\upepsilon_1^2 \cdot x_0+4\upepsilon_1\upepsilon_2\\
&=&x_0^3+3\upepsilon_1 \cdot x_0^2+(4\upepsilon_2+2\upepsilon_1^2)\cdot x_0+4\upepsilon_1\upepsilon_2. 
\end{eqnarray*}
Hence, $B_{\ut_2(\C)}(S^2(\C^2))=\C[f_1,f_2,f_3]$, where
$$f_1:=\upepsilon_1, f_2:=2\upepsilon_2+\upepsilon_1^2, f_3:=\upepsilon_1\upepsilon_2.$$
Clearly, $\C[f_1,f_2,f_3]\subseteq\C[\upepsilon_1,\upepsilon_2]$.  
Conversely, as $\upepsilon_1=f_1\in \C[f_1,f_2,f_3]$ and $$\upepsilon_2=\frac{1}{2}(f_2-f_1^2)\in \C[f_1,f_2,f_3],$$ thus
$\C[\upepsilon_1,\upepsilon_2]\subseteq \C[f_1,f_2,f_3]$, implying that $B_{\ut_2(\C)}(S^2(\C^2))=\C[x_1,x_2]^{S_2}$.

(3) Consider $d=3$.  In this case, note that $[E_{11}]=\dia\{3,2,1,0\}$ and $[E_{22}]=\dia\{0,1,2,3\}$. Thus
the characteristic polynomial is
\begin{eqnarray*}
\upvarphi_{\ut_2(\C)}(S^3(\C^2))&=&(x_0+3x_1)(x_0+2x_1+x_2)(x_0+x_1+2x_2)(x_0+3x_2)\\
&=&(x_0+3x_1)(x_0+3x_2)\cdot(x_0+2x_1+x_2)(x_0+x_1+2x_2)\\
&=&(x_0^2+3\upepsilon_1 \cdot x_0+9\upepsilon_2)\cdot (x_0^2+3\upepsilon_1 \cdot x_0+(\upepsilon_2+2\upepsilon_1^2)),
\end{eqnarray*}
for which the second and third coefficients are: $6\upepsilon_1$ and $10\upepsilon_2+11\upepsilon_1^2$ respectively.
Hence, a similar argument for the case $d=2$ shows that $B_{\ut_2(\C)}(S^3(\C^2))=\C[x_1,x_2]^{S_2}$ as well.
\hbo}\end{exam}

\begin{exam}[$n=3$]{\rm
(1) For the case of $d=1$, the coefficient algebra $B_{\ut_3(\C)}(\C^3)$  is $\C[\upepsilon_1, \upepsilon_2,\upepsilon_3]$ because the characteristic polynomial $\upvarphi_{\ut_3(\C)}(\C^3)=(x_0+x_{1})(x_0+x_{2})(x_0+x_{3}).$ 

(2) Consider the case of $d=2$. The ordered basis $\B=\{e_1^2,e_1e_2,e_1e_3,e_2^2,e_2e_3,e_3^2\}$ and
the characteristic polynomial is
\begin{eqnarray*}
\upvarphi_{\ut_3(\C)}(S^2(\C^3))& = & (x_0+2x_1)(x_0+x_1+x_2)(x_0+x_1+x_3)(x_0+2x_2)(x_0+x_2+x_3)(x_0+2x_3) \\
 & = & (x_0+2x_1)(x_0+2x_2)(x_0+2x_3) \cdot (x_0+x_1+x_2)(x_0+x_1+x_3)(x_0+x_2+x_3)\\
 &=&(x_0^3+2\upepsilon_1 \cdot x_0^2+4\upepsilon_2\cdot x_0+8\upepsilon_3)\cdot (x_0^3+2\upepsilon_1 \cdot x_0^2+(\upepsilon_2+\upepsilon_1^2)\cdot x_0+(\upepsilon_1\upepsilon_2-\upepsilon_3)).
\end{eqnarray*}
Thus, the coefficient algebra $B_{\ut_3(\C)}(S^2(\C^3))\subseteq \C[\upepsilon_1, \upepsilon_2,\upepsilon_3]$. We observe that the first four coefficients $\upvarphi_{\ut_3(\C)}(S^2(\C^3))$ are
$$c_0=1,c_1=4\upepsilon_1,c_2=5\upepsilon_2+5\upepsilon_1^2,c_3=7\upepsilon_3+11\upepsilon_1\upepsilon_2+2\upepsilon_1^3.$$
Hence, $\upepsilon_1=\frac{1}{4}\cdot c_1\in B_{\ut_3(\C)}(S^2(\C^3))$ and so $\upepsilon_2=\frac{1}{5}(c_2-5\upepsilon_1^2)\in B_{\ut_3(\C)}(S^2(\C^3))$.
These facts, together with $\upepsilon_3=\frac{1}{7}(c_3-11\upepsilon_1\upepsilon_2-2\upepsilon_1^3)$
show that $\upepsilon_3$ also belongs to $B_{\ut_3(\C)}(S^2(\C^3))$. 

Therefore, $B_{\ut_3(\C)}(S^2(\C^3))=\C[\upepsilon_1, \upepsilon_2,\upepsilon_3]=\C[x_1,x_2,x_3]^{S_3}$.
 \hbo}\end{exam}

 \section{Coefficient Algebras of $\gl_n(\C)$ and $\sll_n(\C)$}\label{sec4}
\setcounter{equation}{0}
\renewcommand{\theequation}
{4.\arabic{equation}}
\setcounter{theorem}{0}
\renewcommand{\thetheorem}
{4.\arabic{theorem}}
 
 \noindent This section aims to compute the coefficient algebras of  the general linear Lie algebra  $\gl_n(\C)$
 and the special linear Lie algebra $\sll_n(\C)$ on any symmetric power $S^d(\C^n)$ of the standard representation $\C^n$, for all $d\in\N^+$.

Consider the conjugation action of the general linear group $\GL_n(\C)$ on the underlying space of its Lie algebra $\gl_n(\C)$ defined by
\begin{equation}
\label{ }
(g,A)\mapsto g\cdot A\cdot g^{-1}
\end{equation}
for all $g\in \GL_n(\C)$ and $A\in \gl_n(\C)$.  
The following result is well-known in  classical invariant theory; see  \cite[Proposition 4.1.2]{Kra17}. 

\begin{lem}\label{lem4.1}
The invariant ring
\begin{equation}
\label{ }
\C[\gl_n(\C)]^{\GL_n(\C)}=\C[x_{ij}\mid 1\leqslant i,j\leqslant n]^{\GL_n(\C)}=\C[s_1,s_2,\dots,s_n]=\C[\Tr_1, \Tr_2, \dots, \Tr_n]
\end{equation}
is a polynomial algebra of Krull dimension $n$ over $\C$, where $s_i$ $(1\leqslant i\leqslant n)$  denotes the sum of all principal $i\times i$-minors of the generic matrix $X=(x_{ij})_{n\times n}$ and
$$\Tr_i:=\Tr(X^i)$$ 
denotes the trace function of degree $i$ for all $i\in\N^+$.
\end{lem} 
 
 Note that when $n>1$, $s_i=\Tr_i$ if and only if $i=1$, as the following example shows or see \cite[Section 2.4]{KP96} for more explicit calculations. 
 
 \begin{exam}\label{exam4.2}
{\rm (1) Consider $n=2$. The generic matrix
 $$X=\begin{pmatrix}
     x_{11} & x_{12}   \\
     x_{21} & x_{22}  
\end{pmatrix}.$$
Then $s_1=\Tr_1=\Tr(X)=x_{11}+x_{22}$ and $s_2=\det(X)=x_{11}x_{22}-x_{12}x_{21}$.  A direct calculation shows that
$\Tr_2=\Tr(X^2)=x_{11}^2+x_{22}^2+2x_{12}x_{21}$.

(2) If $n=3$, then $$X=\begin{pmatrix}
     x_{11} & x_{12} &x_{13}  \\
     x_{21} & x_{22} & x_{23} \\
     x_{31} & x_{32} & x_{33}
\end{pmatrix}$$
and $s_1=\Tr_1=\Tr(X)=x_{11}+x_{22}+x_{33}$. Furthermore, 
\begin{eqnarray*}
s_2 & = & \sum (\textrm{all }2\times 2\textrm{ principal minors of }X)\\
&=&\det\begin{pmatrix}
     x_{11} & x_{12}   \\
     x_{21} & x_{22}  
\end{pmatrix}+\det\begin{pmatrix}
     x_{11} & x_{13}   \\
     x_{31} & x_{33}  
\end{pmatrix}+\det \begin{pmatrix}
     x_{22} & x_{23}   \\
     x_{32} & x_{33}  
\end{pmatrix}\\
&=& (x_{11}x_{22}-x_{12}x_{21})+(x_{11}x_{33}-x_{13}x_{31})+(x_{22}x_{33}-x_{23}x_{32})
\end{eqnarray*}
and $s_3=\det(X)$.
\hbo}\end{exam}

 \begin{proof}[Proof of Theorem \ref{thm2}] 
Given $n^2+1$ new variables $x_0,x_{ij}$ for $1\leqslant i,j\leqslant n$ and setting a generic matrix $X=(x_{ij})_{n\times n}$, we assume that the characteristic polynomial
of $\gl_n(\C)$ with respect to the standard basis  $\{E_{ij}\mid 1\leqslant i,j\leqslant n\}$ on $S^d(\C^n)$ is
\begin{eqnarray*}
\upvarphi_{\gl_n(\C)}(S^d(\C^n))&=&\det\left(x_0\cdot I+\sum_{1\leqslant i,j\leqslant n} x_{ij}\cdot [E_{ij}]\right)\\
&=&\det\left(x_0\cdot I+[X]\right)\\
&=&x_0^m+\sum_{i=1}^m z_i\cdot x_0^{m-i},
\end{eqnarray*}
where $m:=\binom{n+d-1}{d}$ and $z_1,\dots,z_m\in \C[x_{ij}\mid 1\leqslant i,j\leqslant n]$. We may extend the conjugacy action of $\GL_n(\C)$ on 
$\C[x_{ij}\mid 1\leqslant i,j\leqslant n]$ to an action of $\GL_n(\C)$ on $\C[x_0,x_{ij}\mid 1\leqslant i,j\leqslant n]$ by setting $g\cdot x_0=x_0$ for all $g\in \GL_n(\C)$. Then the characteristic polynomial $\upvarphi_{\gl_n(\C)}(S^d(\C^n))$ is invariant under this extended conjugacy action of $\GL_n(\C)$. This means that each coefficient $z_i$ is a conjugacy invariant of $\GL_n(\C)$. Hence,
$B_{\gl_n(\C)}(S^d(\C^n))=\C[z_1,z_2,\dots,z_m]$ is contained in 
$\C[x_{ij}\mid 1\leqslant i,j\leqslant n]^{\GL_n(\C)}$; note that the latter invariant ring is equal to $\C[s_1,s_2,\dots,s_n]=\C[\Tr_1, \Tr_2, \dots, \Tr_n]$ by Lemma \ref{lem4.1}.

To complete the proof, we need to show that $\C[x_{ij}\mid 1\leqslant i,j\leqslant n]^{\GL_n(\C)}\subseteq B_{\gl_n(\C)}(S^d(\C^n)).$ In other words, it suffices to show that $\Tr_i\in B_{\gl_n(\C)}(S^d(\C^n))$ for $i=1,2,\dots,n$.
Firstly, we note that $\Tr_1=x_{11}+\cdots+x_{nn}$. As $z_1$ is a nonzero $\GL_n(\C)$-invariant of degree $1$, thus 
$z_1=a_1\cdot \Tr_1$ for some nonzero $a_1\in\C$. Hence, $\Tr_1=a_1^{-1}\cdot z_1\in B_{\gl_n(\C)}(S^d(\C^n))$.

Secondly, we endow the graded lexicographic monomial ordering on $\C[x_{ij}\mid 1\leqslant i,j\leqslant n]$ with
$x_{11}>x_{22}>\cdots>x_{nn}>\cdots$. We \textit{claim} that the leading monomial of $z_i$ is $x_{11}^i$ for $i=1,2,\dots,n$. To see that, we consider the standard surjective $\C$-algebra homomorphism 
\begin{equation}
\label{eq4.3}
\uppi:\C[x_0,x_{ij}\mid 1\leqslant i,j\leqslant n]\ra \C[x_0,x_{ij}\mid 1\leqslant i\leqslant j\leqslant n]
\end{equation}
defined by $x_{ij}\mapsto 0$ for all $i>j$ and fixing other $x_{ij}$ and $x_0$. The image of the characteristic polynomial 
$\upvarphi_{\gl_n(\C)}(S^d(\C^n))$ under $\uppi$ is equal to the characteristic polynomial 
$\upvarphi_{\ut_n(\C)}(S^d(\C^n))$. In particular, the image of the coefficient $z_i$ in $\upvarphi_{\gl_n(\C)}(S^d(\C^n))$
is the coefficient $c_i$ of $\upvarphi_{\ut_n(\C)}(S^d(\C^n))$ in (\ref{eq3.12}). 
We have seen in the proof of Theorem \ref{thm1}  that  the leading monomial of $c_i$ is $x_{11}^i$. Thus, the leading monomial of $z_i$ must be $x_{11}^i$ and the claim follows.

Since the leading monomial of $\Tr_i$ is $x_{11}^i$, there are a nonzero $a_i\in\C$ and a polynomial $Q$ in $i-1$ variables such that
\begin{equation}
\label{eq4.4}
z_i=a_i\cdot\Tr_i+Q(\Tr_1,\dots,\Tr_{i-1})
\end{equation}
for all $i\geqslant 2$. In particular, $z_2-a_2\cdot\Tr_2\in\C[\Tr_1]$ and 
$$\Tr_2\in\C[\Tr_1,z_2]=\C[z_1,z_2]\subseteq B_{\gl_n(\C)}(S^d(\C^n)).$$
Proceeding in the same way and using (\ref{eq4.4}) repeatedly, we see that
$$\Tr_3,\dots,\Tr_n\in \C[z_1,z_2,\dots,z_n]\subseteq B_{\gl_n(\C)}(S^d(\C^n)).$$

Therefore, $B_{\gl_n(\C)}(S^d(\C^n))=\C[x_{ij}\mid 1\leqslant i,j\leqslant n]^{\GL_n(\C)}$ and the proof is complete.
\end{proof}

 \begin{rem}{\rm
Note that the image of $B_{\gl_n(\C)}(S^d(\C^n))$ under the homomorphism $\uppi$ in (\ref{eq4.3}) is exactly the coefficient algebra $B_{\ut_n(\C)}(S^d(\C^n))$. Theorem \ref{thm1}, together with the algebraic geometry method appeared in \cite[Example 2.1.3]{DK15} also can give a proof to Theorem \ref{thm2}.
\hbo}\end{rem}

\begin{exam}[$n=2$]{\rm
The generic matrix is
$$X=\sum_{1\leqslant i,j\leqslant 2}x_{ij}\cdot E_{ij}=\begin{pmatrix}
     x_{11} & x_{12}   \\
      x_{21}& x_{22} 
\end{pmatrix}$$
and consider the ordered basis $\B$ in (\ref{eq3.4}).

(1) Suppose that $d=2$. By Example \ref{exam3.1}, we see that the resulting matrix of  $X$ is
$$[X]=\begin{pmatrix}
   2x_{11}   & 2x_{12} &0  \\
    x_{21}  & x_{11}+x_{22} &x_{12} \\
    0& 2x_{21}&2x_{22} 
\end{pmatrix}.$$
Thus the characteristic polynomial is
\begin{eqnarray*}
\varphi_{\gl_2(\C)}(S^2(\C^2)) & = &  \det\begin{pmatrix}
   x_0+2x_{11}   & 2x_{12} &0  \\
    x_{21}  & x_0+x_{11}+x_{22} &x_{12} \\
    0& 2x_{21}&x_0+2x_{22} 
\end{pmatrix}\\
 & = & x_0^3+z_1\cdot x_0^2+z_2\cdot x_0+z_3 
\end{eqnarray*}
where $z_1 = 3x_{11} + 3x_{22}, z_2=2x_{11}^2 - 4x_{12}x_{21} + 8x_{1}x_{22} + 2x_{22}^2$, and 
$$z_3=- 4x_{11}x_{12}x_{21} + 4x_{11}^2x_{22} - 4x_{12}x_{21}x_{22} + 4x_{11}x_{22}^2.$$
Recall that the conjugation invariants $\Tr_1=\Tr(X)$ and $\Tr_2=\Tr(X^2)$ were given in Example \ref{exam4.2} (1).  A direct computation verifies the following relations among $\{\Tr_1,\Tr_2,z_1,z_2,z_3\}$:
\begin{eqnarray*}
z_1& = & 3\cdot\Tr_1  \\
z_2 & = & 4\cdot\Tr_1^2-2\cdot\Tr_2\\
z_3&=& 2\cdot \Tr_1^3 - 2\cdot \Tr_1\cdot\Tr_2.
\end{eqnarray*}
Hence, the coefficient algebra $B_{\gl_2(\C)}(S^2(\C^2)) =\C[\gl_2(\C)]^{\GL_2(\C)}.$

(2) Suppose that $d=3$. The  resulting matrix of  $X$ on $S^3(\C^2)$ is
$$[X]=\begin{pmatrix}
   3x_{11}   & 3x_{12} &0  &0\\
    x_{21}  & 2x_{11}+x_{22} &2x_{12} &0\\
    0& 2x_{21}&x_{22}+ 2x_{22}&x_{12}\\
    0&0&3x_{21}&3x_{22}
\end{pmatrix}$$
and so the characteristic polynomial in this case is
\begin{eqnarray*}
\varphi_{\gl_2(\C)}(S^3(\C^2))&=&\det\begin{pmatrix}
   x_0+3x_{11}   & 3x_{12} &0  &0\\
    x_{21}  & x_0+2x_{11}+x_{22} &2x_{12} &0\\
    0& 2x_{21}&x_0+x_{11}+ 2x_{22}&x_{12}\\
    0&0&3x_{21}&x_0+3x_{22}
\end{pmatrix}\\
&=&x_0^4+z_1\cdot x_0^3+z_2\cdot x_0^2+z_3\cdot x_0+c_z 
\end{eqnarray*}
where 
\begin{eqnarray*}
z_1 & = & 6(x_{11} + x_{22}) \\
z_2 & = & 11x_{11}^2 - 10x_{12}x_{21} + 32x_{11}x_{22} + 11x_{22}^2\\ 
z_3&=&6x_{11}^3 - 30x_{11}x_{12}x_{21} + 48x_{11}^2x_{22} - 30x_{12}x_{21}x_{22} +48x_{11}x_{22}^2 + 6x_{22}^3\\
z_4&=&-18x_{11}^2x_{12}x_{21} + 9x_{12}^2x_{21}^2 + 18x_{11}^3x_{22} - 54x_{11}x_{12}x_{21}x_{22} + 
45x_{11}^2x_{22}^2 \\
&&- 18x_{12}x_{21}x_{22}^2 + 18x_{11}x_{22}^3.
\end{eqnarray*}
The following relations 
\begin{eqnarray*}
z_1 & = & 6\cdot\Tr_1\\
z_2 & = & 16\cdot(\Tr_1)^2-5\cdot\Tr_2\\
z_3&=&21\cdot(\Tr_1)^3-15\cdot\Tr_1\cdot\Tr_2\\
4z_4&=&45\cdot(\Tr_1)^4 - 54\cdot(\Tr_1)^2\cdot\Tr_2 + 9\cdot(\Tr_2)^2
\end{eqnarray*}
show that $B_{\gl_2(\C)}(S^3(\C^2))=\C[\gl_2(\C)]^{\GL_2(\C)}$ as well.
\hbo}\end{exam}

The rest of this section is to calculate the coefficient algebra $B_{\sll_n(\C)}(S^d(\C^n))$ of  $\sll_n(\C)$ on
$S^d(\C^n)$.
Let us fix
$$\{E_{ij}, E_{ss}-E_{s+1,s+1}\mid 1\leqslant i\neq j\leqslant n, 1\leqslant s\leqslant n-1\}$$
as a basis of  $\sll_n(\C)$. Then the generic matrix 
\begin{eqnarray}\label{eq4.5}
X&=&\sum_{s=1}^{n-1}x_{ss}\cdot (E_{ss}-E_{s+1,s+1})+\sum_{1\leqslant i\neq j\leqslant n}x_{ij}\cdot E_{ij} \nonumber\\
 & = & \begin{pmatrix}
    x_{11} &x_{12}&x_{13}&\cdots&x_{1n}    \\
    x_{21}  &x_{22}-x_{11}&x_{23}&\cdots&x_{2n}\\  
    \vdots&\cdots&\ddots&\ddots&\vdots\\
       \vdots&\cdots&\cdots&x_{n-1,n-1}-x_{n-2,n-2}&x_{n-1,n}\\
    x_{n1}&\cdots&\cdots&x_{n,n-1}&-x_{n-1,n-1}
\end{pmatrix}.
\end{eqnarray}
and  the characteristic polynomial of $\sll_n(\C)$ on $S^d(\C^n)$ is
\begin{equation} \label{eq4.6}
\begin{aligned}
\upvarphi_{\sll_n(\C)}(S^d(\C^n))~& =~ \det\left(x_0\cdot I+[X]\right)\\
&=~x_0^m+w_1\cdot x_0^{m-1}+w_2\cdot x_0^{m-2}+\cdots+w_{m-1}\cdot x_0+w_m
\end{aligned}
\end{equation}
where $m:=\binom{n+d-1}{d}$ and $w_1,w_2,\dots,w_m\in\C[x_{ij},x_{ss}\mid 1\leqslant i\neq j\leqslant n, 1\leqslant s\leqslant n-1]$. 

\begin{proof}[Proof of Theorem \ref{thm3}]
We consider the standard surjective $\C$-algebra homomorphism 
\begin{equation}
\label{eq4.7}
\uprho:\C[x_0,x_{ij}\mid 1\leqslant i,j\leqslant n]\ra \C[x_0,x_{ij},x_{ss}\mid 1\leqslant i\neq j\leqslant n, 1\leqslant s\leqslant n-1]
\end{equation}
defined by 
$$x_{22}\mapsto x_{22}-x_{11},\dots,x_{n-1,n-1}\mapsto x_{n-1,n-1}-x_{n-2,n-2},x_{nn}\mapsto -x_{n-1,n-1}$$
and fixing other $x_{ij}$ and $x_0$. Clearly, the image of $\upvarphi_{\gl_n(\C)}(S^d(\C^n))$ is 
$\upvarphi_{\sll_n(\C)}(S^d(\C^n))$ and thus
\begin{equation}
\label{ }
\uprho(z_i)=w_i
\end{equation}
where $i=1,2,\dots,m$ and these $z_i$ are coefficients of $\upvarphi_{\gl_n(\C)}(S^d(\C^n))$.
As a consequence, the coefficient algebra $B_{\sll_n(\C)}(S^d(\C^n))$ is the image of $B_{\gl_n(\C)}(S^d(\C^n))$
under the homomorphism $\uprho$. For $i\in\{1,2,\dots,n\}$, we write $\tr_i:=\uprho(\Tr_i)$. By Theorem \ref{thm2},
we see that
$$B_{\sll_n(\C)}(S^d(\C^n))=\C\left[\uprho(\Tr_1),\uprho(\Tr_2),\dots,\uprho(\Tr_n)\right]=\C[\tr_1,\tr_2,\dots,\tr_n].$$
Note that $\Tr_1=x_{11}+x_{22}+\cdots+x_{nn}$, thus $\tr_1=x_{11}+(x_{22}-x_{11})+\cdots+(x_{n-1,n-1}-x_{n-2,n-2})+(-x_{n-1,n-1})=0.$ This implies that
$$B_{\sll_n(\C)}(S^d(\C^n))=\C[\tr_2,\tr_3,\dots,\tr_n].$$

On the other hand,  the conjugacy action of the special linear group $\SL_n(\C)$ on the underlying space of $\sll_n(\C)$ induces a $\C$-algebraic action of $\SL_n(\C)$ on the polynomial ring $\C[\sll_n(\C)]$, and the invariant ring
$\C[\sll_n(\C)]^{\SL_n(\C)}$ is minimally generated by $\{\tr_2,\tr_3,\dots,\tr_n\};$
see for example, \cite[Chapter VIII, Exercise 13]{Bou05}. Therefore, $B_{\sll_n(\C)}(S^d(\C^n))=\C[\tr_2,\tr_3,\dots,\tr_n]=\C[\sll_n(\C)]^{\SL_n(\C)}.$
\end{proof}

Theorem \ref{thm3}, together with the relationship between $\{s_1,s_2,\dots,s_n\}$ and $\{\Tr_1,\Tr_2,\dots,\Tr_n\}$ (see for example, \cite[Section 8.1.4]{Car17}), obtains the following result. 

\begin{coro}\label{coro4.5}
The characteristic polynomial of the special linear Lie algebra $\sll_n(\C)$ with respect to the standard representation $\C^n$ is
\begin{equation}
\label{ }
\upvarphi_{\sll_n(\C)}(\C^n)= x_0^n+\sum_{i=2}^n w_i\cdot  x_0^{n-i}
\end{equation}
where 
\begin{equation}
\label{ }
w_i:=\frac{1}{i!}\det\begin{pmatrix}
     0 & i-1&0&\cdots&0   \\
     \tr_2 &  0&i-2&\ddots&\vdots\\
     \tr_3&\tr_2&\ddots&\ddots&0\\
     \vdots&\ddots&\ddots&\ddots&1\\
     \tr_i&\cdots&\tr_3&\tr_2&0
\end{pmatrix}.
\end{equation}
\end{coro}

We provide the following examples to illustrate how to compute $\upvarphi_{\sll_n(\C)}(\C^n)$.

\begin{exam}{\rm Suppose $n=2$. Then $X=\begin{pmatrix}
   x_1   &  x_{12}  \\
   x_{21}   & -x_1 
\end{pmatrix}$. Clearly, $\tr_1=\Tr(X)=0$ and since
$$X^2=\begin{pmatrix}
   x_1   &  x_{12}  \\
   x_{21}   & -x_1 
\end{pmatrix}\begin{pmatrix}
   x_1   &  x_{12}  \\
   x_{21}   & -x_1 
\end{pmatrix}=\begin{pmatrix}
   x_1^2+x_{12}x_{21}   &  0  \\
   0  & x_1^2+x_{12}x_{21} 
\end{pmatrix}$$
it follows that $\tr_2=\Tr(X^2)=2(x_1^2+x_{12}x_{21})$. Therefore, the characteristic polynomial is
$$\upvarphi_{\sll_2(\C)}(\C^2)= x_0^2-\frac{\tr_2}{2!}=x_0^2-(x_1^2+x_{12}x_{21}).$$
Compare this statement with the results in Section \ref{sec2.1}. \hbo
}\end{exam}

\begin{exam}{\rm Suppose $n=3$. Then $X=\begin{pmatrix}
   x_1   &  x_{12} & x_{13} \\
   x_{21}   & x_2-x_1 &x_{23}\\
   x_{31} &x_{32}&-x_2
\end{pmatrix}$ and $\tr_1=\Tr(X)=0$.  A direct calculation shows that 
$$X^2=\begin{pmatrix}
x_{12}x_{21} + x_{13}x_{31} +  x_1^2 &  x_{13}x_{32} + x_{12}x_2&   x_{13}x_1 - x_{13}x_2 + x_{12}x_{23}    \\
x_{23}x_{31} + x_{21}x_2      &  x_{23}x_{32} + x_{12}x_{21} + x_1^2 - 2x_1x_2 + x_2^2& x_{13}x_{21} - x_{23}x_1\\
x_{31}x_1 - x_{31}x_2 + x_{32}x_{21}& x_{12}x_{31} - x_{32}x_1 & x_{13}x_{31} + x_{23}x_{32} + x_2^2
\end{pmatrix}.$$
Thus $\tr_2=\Tr(X^2)=2(x_{12}x_{21}+x_{13}x_{31} + x_{23}x_{32} + x_1^2 - x_1x_2 + x_2^2).$ Similarly, computing $X^3$ obtains
$$\tr_3=\Tr(X^3)=3 (x_1x_{13}x_{31}-x_2x_{13}x_{31}+x_{12}x_{23}x_{31}+x_{13}x_{21}x_{32}-x_{23}x_{32}x_1+x_{12}x_{21}x_2+x_1^2x_2-x_1x_2^2).$$
Hence, the characteristic polynomial is 
\begin{eqnarray*}
\upvarphi_{\sll_3(\C)}(\C^3)& = & x_0^3+\frac{1}{2!}\det\begin{pmatrix}
  0    &  1  \\
   \tr_2   &  0
\end{pmatrix}\cdot x_0+\frac{1}{3!}\det\begin{pmatrix}
  0    &  2  &0\\
   \tr_2   &  0&1\\
   \tr_3&\tr_2&0
\end{pmatrix}
\end{eqnarray*}
which could be simplified as $x_0^3-\frac{\tr_2}{2}\cdot x_0+\frac{\tr_3}{3}.$
\hbo}\end{exam}

\vspace{2mm}
\noindent \textbf{Acknowledgements}. 
This research was partially supported by the University of Saskatchewan (Grant No. APEF-121159) and NSFC (Grant No. 12561003). The authors would like to thank the anonymous referee and the editor-in-chief for their careful reading, valuable suggestions and helpful  comments.

%\vspace{2mm}
%\noindent \textbf{Data availability}. Data sharing is not applicable to this article as no datasets were generated or analysed during the current study. 

%\vspace{2mm}
%\noindent \textbf{Conflict of interest}. On behalf of all authors, the corresponding author states that there is no conflict of interest.

%%%%%%%%%%%%%%%%%%%%%%%%%%References%%%%%%%%%%%%%%%%%%%%%%%%

\raggedright
\end{document}